\documentclass[11pt]{amsart}
\usepackage{color}

\newcommand{\bt}{\begin{Theorem}}
\newcommand{\et}{\end{Theorem}}
\newcommand{\bi}{\begin{itemize}}
\newcommand{\ei}{\end{itemize}}
\newcommand{\bea}{\begin{eqnarray}}
\newcommand{\ba}{\begin{array}}
\newcommand{\eea}{\end{eqnarray}}
\newcommand{\ea}{\end{array}}

\newcommand{\supp}{\mbox{supp}}

\newtheorem{Definition}{Definition}[section]
\newtheorem{Theorem}[Definition]{Theorem}
\newtheorem{Lemma}[Definition]{Lemma}

\newtheorem*{theoremA*}{Theorem A}
\newtheorem*{theoremB*}{Theorem B}
\newtheorem*{lemmaA*}{Lemma A}
\newtheorem*{lemmaB*}{Lemma B1}
\newtheorem*{lemmaB1*}{Lemma B2}
\newtheorem*{Proofofmainthm*}{Proof of main theorem}
\newcommand{\be}{\begin{equation}}
\newcommand{\ee}{\end{equation}}
\newcommand{\bvb}{\begin{verbatim}}

\newcommand{\evb}{\end{verbatim}}

\newcommand{\R}{\mathbb R}%
\newcommand{\C}{\mathbb C}%
\newcommand{\Z}{\mathbb Z}%
\newcommand{\x}{\mathfrak X}%

\allowdisplaybreaks[1]

\renewcommand{\Re}{\mbox{Re }}
\begin{document}
\baselineskip16pt

\author[Pratyoosh Kumar and Sumit Kumar Rano]{Pratyoosh Kumar and Sumit Kumar Rano }
\address{Department of Mathematics, Indian Institute of Technology Guwahati, 781039, India.
E-mail: pratyoosh@iitg.ac.in and s.rano@iitg.ac.in}

\title[Theorem of Roe and Strichartz on homogeneous trees ]
{A Theorem of Roe and Strichartz on homogeneous trees}

\begin{abstract}  In 1980, J. Roe proved that if $\{f_{k}\}_{k\in\mathbb{Z}}$ is doubly infinite sequence of functions in $\mathbb{R}$ which is uniformly bounded and satisfies $(df_{k}/dx)=f_{k+1}$ for all $k\in\mathbb{Z}$ then $f_{0}(x)=a\sin(x+\theta)$ for some $a,\theta\in\mathbb{R}$. Later in 1993 Strichartz suitably extended the above result to $\mathbb{R}^n$. In this article we prove a version of their result for homogeneous trees.
\end{abstract}

\maketitle

\section{Introduction}
Let $\{f_{k}\}_{k\in\mathbb{Z}}$ be a doubly infinite sequence of real-valued functions of a real variable with
	$\frac{d}{dx}f_{k}=f_{k+1}.$ In 1980, J. Roe \cite{R} proved that
	if there exists a constant $M>0$ such that $|f_{k}(x)|\leq M\text{ for all }n\text{ and }x$
	then $f_{0}(x)=a\sin(x+\theta)$ for some $a,\theta\in\mathbb{R}$. Many generalization of this result is available in the literature (see \cite{HR1}, \cite{HR2}).
In 1993, Strichartz \cite{S} extended the above result to $\mathbb{R}^n$ by substituting $d/dx$ with the Laplacian $\Delta_{\mathbb{R}^{n}}$ on $\mathbb{R}^n$. Strichartz's result can be stated as follows.
\begin{Theorem}[Strichartz]\label{th2}
	Let $\{f_{k}\}_{k\in\mathbb{Z}}$ be a doubly infinite sequence of functions in $\mathbb{R}^{n}$ satisfying $\Delta_{\mathbb{R}^{n}} f_{k}=f_{k+1}$ for all $k\in\mathbb{Z}$ and $|f_{k}(x)|\leq M$ for all $k\in\mathbb{Z}$ and $x\in\mathbb{R}^{n}$, where $M$ is a real number. Then $\Delta_{\mathbb{R}^{n}} f_{0}=-f_{0}$.
\end{Theorem}
Furthermore, Strichartz also proved that the above result holds for Heisenberg groups but fails for hyperbolic 3-space.
It turns out that the negative result on hyperbolic 3-space can indeed be extended to homogeneous trees ( which can be considered as a discrete version of hyperbolic spaces). A homogeneous tree $\mathfrak{X}$ of degree $q+1$ is a connected graph with no loops such that every vertex is adjacent to $q+1$ other vertices. For details about notation and preliminary results, we refer to section 2. Henceforth we assume $q\geq 2$. The distance $d(x,y)$ between two vertices $x$ and $y$  defined as the number of edges joining $x$ and $y$. The natural Laplace operator (or Laplacian) $\mathcal L$ on $\x$ is defined by
\begin{equation}\label{eq00}\mathcal Lf(x)=f(x)-\frac{1}{q+1}\sum\limits_{y:d(x,y)=1}f(y).\end{equation}

Consider the spherical function $\phi_z$ which is a radial eigenfunction of the Laplacian with eigenvalue $\gamma(z),$ where  $\gamma$ is an analytic function defined by the formula
\begin{equation}\label{eq01}
	\gamma(z)=1-\frac{q^{1/2+iz}+q^{1/2-iz}}{q+1}.
\end{equation}
 Note that the image of $S_1=\{z\in\C: |\Im z|\leq1/2\}$ under the map $\gamma$ is an elliptic region which intersects $\{w\in\mathbb{C}:|w|=1\}$ in infinitely many points. Assume $f_{k}(x)=\gamma(z_{1})^{k}\phi_{z_1}(x)+\gamma(z_{2})^{k}\phi_{z_2}(x)$ for some $z_1,z_2$ in $S_1$ such that $\gamma(z_1)\neq\gamma(z_2)$ and $|\gamma(z_1)|=|\gamma(z_2)|=1$. Since $\phi_z$ is uniformly bounded on $S_1$, $\{f_k\}$ satisfy all the hypothesis of the Theorem \ref{th2}, but  $f_{0}$ fails to be an eigenfunction of $\mathcal{L}$.

A careful analysis of the above counterexample reveals that the failure of Strichartz's result on $\mathfrak{X}$ is mainly due to the spectrum of $\mathcal{L}$.
It was also observed in \cite{KRS2} that failure of Strichartz results is rooted in the $p$-dependence of the $L^p-$spectrum of the Laplacian on the hyperbolic spaces. In \cite{KRS2} it was proved that the theorem indeed remains valid when uniform boundedness is replaced by uniform ``almost $L^p$ boundedness". Here it is worth mentioning that these size estimates arise naturally due to the behaviour of the Poisson transforms, which also acts as the eigenfunctions of $\mathcal{L}$ with the eigenvalues lying in the interior of the ellipse (\ref{equation 1.2}) (For details see \cite{KR}).
The version of Roe's theorem that we have proved in this article in the context of homogeneous trees are the following. A similar result is also proved  for the harmonic $NA$ groups and symmetric spaces (see \cite{KRS2},\cite{RS}).
\begin{theoremA*}\label{Theorem A}
	Let $f$ be a measurable on $\mathfrak{X}$ and $z\in\R\setminus(\tau/2)\Z$. If there exists an $M>0$ such that $\|\mathcal{L}^{k}f\|_{L^{2,\infty}(\mathfrak{X})}\leq M|\gamma(z)|^{k}$ for all $k\in\mathbb{Z}$ then $\mathcal{L}f\equiv\gamma(z)f$. In particular, there exists $F\in L^{2}(\Omega)$ such that $f\equiv\mathcal{P}_{z}F$.
\end{theoremA*}
Since $\gamma(z)\in \R$ whenever $z\in\R\setminus(\tau/2)\Z$, Theorem A can be thought of as a suitable extension of  Strichartz's result on homogeneous tree.  In particular if we define $f_{k}=\gamma(z)^{-k}\mathcal{L}^{k}f$, then the statement of Theorem A resembles Theorem \ref{th2} with the only difference that the $L^{\infty}$ boundedness is being replaced by weak $L^{2}$ boundedness.
\begin{theoremB*}\label{Theorem B}
	Let $f$ be a measurable on $\mathfrak{X}$ and $1<p<2$.
	\begin{enumerate}
		\item Suppose that $z=n\tau+i\delta_{p'}$ for some $n\in\mathbb{Z}$. If there exists an $M>0$ such that $\|\mathcal{L}^{k}f\|_{L^{p^{\prime},\infty}(\mathfrak{X})}\leq M|\gamma(z)|^{k}$ for all $k\in\mathbb{Z}_{+}$ then $\mathcal{L}f\equiv\gamma(z)f$.
		\item Suppose that $z=(2n+1)\tau/2+i\delta_{p'}$ for some $n\in\mathbb{Z}$. If there exists an $M>0$ such that $\|\mathcal{L}^{-k}f\|_{L^{p^{\prime},\infty}(\mathfrak{X})}\leq M|\gamma(z)|^{-k}$ for all $k\in\mathbb{Z}_{+}$ then $\mathcal{L}f\equiv\gamma(z)f$.
	\end{enumerate}

	In either of these cases, there exists $F\in L^{p^{\prime}}(\Omega)$ such that $f\equiv\mathcal{P}_{z}F$.
\end{theoremB*}

It was proved in \cite{KR} that any weak $L^p$ eigenfunction of the Laplacian of $\mathfrak{X}$ can be represented by the Poisson transform of a $L^{p}$ function on the boundary.
Therefore the conclusion of these theorems is more precise than that of Theorem \ref{th2}.

 We conclude this section by summarizing the contents of this article. In Section 2 we discuss some basic notation, definition, and a few well-known results on $\mathfrak{X}$. In Section 3  we shall provide a detailed proof of the Theorem A and Theorem B. In the last section we shall discuss the sharpness of our main results and provide an outline of the Roe's result separately on $\mathbb{Z}$ (i.e., the case when $q=1$). To make this article self-contained, we have also included a small appendix about the isomorphism theorem at the end of the paper.


\section{Preliminaries}\label{Section 2}

For notation and some preliminary result about homogeneous tree and their group of isometries, we will mainly follow \cite{JF,CMS,CMS1,FTC,FTP2}. Most of our other notation are standard. The letters $\mathbb{Z}_{+},\mathbb{Z},\mathbb{R}$ and $\mathbb{C}$ will respectively denote the set of all non-negative integers, integers, real numbers and complex numbers. For $z\in \C$ we use the notation $\Re z$ and $\Im z$ for real and imaginary part of $z$ respectively. We also need some basic facts about the Lorentz spaces that can be found in \cite{LG1}.

 Let $G$ be the group of isometries of the metric space $(\x,d)$ and let $K$ be the stabilizer of $o$ in $G$. The map $g\rightarrow g\cdot o$ identifies $\x$ with the coset space $G/K$, so that functions on $\x$ corresponds to $K$-right invariant functions on $G$.  Further radial functions on $\x$ corresponds to $K$-bi-invariant functions on $G$. If $E(\x)$ is a function space on $\x$  we will denote by $E(\x)^\#$  the radial functions in $E(\x)$.
An infinite geodesic ray $\omega$ in $\x$ is an one-sided sequence $\{\omega_{n}: n=0,1,2\ldots \}$ where $\omega_n$'s are in $\x$. Let $o$ be some arbitrarily fixed point in $\x$. The boundary of $\x$ is the set of all infinite geodesic rays starting at $o$ and will be denoted by $\Omega$.  Notice that the map $k\rightarrow k\cdot\omega_{0}$ represents a transitive action of $K$ on $\Omega$. Let $\nu$ be the $G$-quasi-invariant probability measure on the boundary $\Omega$. The Poisson kernel $p(g\cdot o,\omega)$ is the Radon-Nikodym derivative $d\nu(g^{-1}\omega)/d\nu(\omega)$ and explicitly written as
$$ p(x,\omega)=q^{h_{\omega}(x)} \;\; \forall x\in \x \;\; \forall \omega\in\Omega,$$
where $h_{\omega}(x)$ is the height of $x$ in $\x$ with respect to $\omega$ (see \cite{FTC} for details). The Poisson transformation $\mathcal{P}_{z}:C(\Omega)\rightarrow C(\x)$ is given by the formula
$$\mathcal{P}_{z}\eta(x)=\int\limits_{\Omega}p^{1/2+iz}(x,\omega)\eta(\omega)d\nu(\omega).$$
It is obvious that $\mathcal{P}_{z}=\mathcal{P}_{z+\tau}$, where $\tau=2\pi/\log q$. We denote the torus $\R/\tau \Z$ by $\mathbb{T}$, which can be identified with the interval $[-\tau/2,\tau/2)$. Let $\mathcal L$ be the Laplacian on $\mathfrak X$ defined in (\ref{eq00}). It is a well-known fact that $\mathcal L \mathcal{P}_{z}\eta(x) =\gamma(z) \mathcal{P}_{z}\eta(x)$ for every $z$ in $\mathbb{C},$ where expression of $\gamma(z)$ is given in (\ref{eq01}). The elementary spherical function $\phi_z$ on $\mathfrak X$ is the Poisson transform of the constant function $\mathbf{1}$.
Note that $\phi_z$ is the radial eigenfunction of $\mathcal L$ with eigenvalue $\gamma(z)$ such that  $\phi_z(o)=1$.
 For a suitable function $f$ on $\mathfrak{X}$, its radialization $\varepsilon f$ is defined as
\begin{equation}\label{Equation 2.3}
	\varepsilon f(x)=\int\limits_{K}f(k\cdot x)dk,
\end{equation}
where $dk$ is the normalized measure on $K$. Some useful facts about radialization:
 \begin{enumerate}
	\item $\|\varepsilon f\|_{p,q}\leq\|f\|_{p,q}$ whenever $1<p<\infty$, $1\leq q\leq\infty$.
\item The operator  $\varepsilon$  commutes with the Laplacian $\mathcal{L}$, that is $\mathcal{L}(\varepsilon f)=\varepsilon(\mathcal{L}f).$
\item Also if $\mathcal L u=\gamma(z)u$ then $\varepsilon u (x)=u(o)\phi_z(x)$.
\end{enumerate}
The following expression of the function $\phi_z$ is well-known (see \cite{FTP1})
\begin{equation}\label{eqsf}
\phi_z(x)= \begin{cases}
\vspace*{.2cm} \left(\frac{q-1}{q+1}|x|+1\right)q^{-|x|/2}&\forall z\in\ \tau\Z\\
\vspace*{.2cm}\left(\frac{q-1}{q+1}|x|+1\right)q^{-|x|/2}(-1)^{|x|}&\forall z\in {\tau/2}+\tau\Z\\
\mathbf{c}(z)q^{{(iz-1/2)}|x|}+\mathbf{c}(-z)q^{{(-iz-1/2)}|x|}&\forall z\in\C\setminus(\tau/2)\Z,
\end{cases}
\end{equation}
where $\mathbf{c}$ is a meromorphic function given by
$$\mathbf{c}(z)=\frac{q^{1/2}}{q+1}\frac{q^{1/2+iz}-q^{-{1/2}-iz}}{q^{iz}-q^{-iz}}\quad\forall z\in (\tau/2)\Z.$$
It is easy to verify that $|\phi_{z}(x)|\leq 1$ for all $x\in\x$ whenever $z\in S_1=\{z\in\C: |\Im z|\leq1/2\}$.
Now we give some $L^p-$type estimates of $\phi_{z}(x)$. For $p\in(1,\infty)$ we define
$$\delta_p=\frac{1}{p}-\frac{1}{2}\;\;\; \text{and}\;\;\; S_p=\{z\in\C: |\Im z|\leq|\delta_p|\}.$$
It is important to note that $\delta_p=-\delta_{p^{\prime}}$ and $S_{2}=\mathbb{R}$. We assume $\delta_1=-\delta_\infty=1/2$ so that $S_1=\{z\in\C: |\Im z|\leq1/2\}$. We shall henceforth write $S_p^\circ$ and $\partial{S_p}$ to denote the usual interior and boundary of $S_p$ respectively.  The following norm estimates of $\phi_{z}$ can be derived by using  (\ref{eqsf}) (see \cite{FTP1} for details).
\begin{Lemma}\label{phiz}
	Let $1<p<2$. Then
	\begin{enumerate}
		\item[$1.$] $\phi_z\in L^{p'}(\x)\quad\text{if and only if}\quad z\in S_p^\circ$
		\item[$2.$] $\phi_z\in L^{p',\infty}(\x)\quad\text{if and only if}\quad z\in S_p.$
		\item[$3.$] $\phi_z\notin L^{2,\infty}(\x)$ if $z\in(\tau/2)\Z$ and $\phi_z\in L^{2,\infty}(\x)$ if $z\in\R\setminus(\tau/2)\Z$.
	\end{enumerate}
\end{Lemma}
From the above lemma it is clear that $\gamma(S_p^\circ)$ lies in the point spectrum of the $\mathcal L,$ which is a bounded operator on $L^{p^\prime}(\mathfrak X)$.
The following observation given in \cite[page 4275]{CMS1} clarifies the spectrum $\sigma_{p}(\mathcal{L})$ of the Laplacian $\mathcal{L}$.

\noindent{\it	For every $p\in[1,\infty]$, the $L^{p}$-spectrum $\sigma_{p}(\mathcal{L})$ of $\mathcal{L}$ is the image of $S_p$ under the map $\gamma$, which is precisely the set of all $w$ in $\mathbb{C}$ which satisfies
	\begin{equation}\label{equation 1.2}
	\left[\frac{1-\Re(w)}{b\cosh(\delta_{p}\log q)}\right]^{2}+\left[\frac{\Im(w)}{b\sinh(\delta_{p}\log q)}\right]^{2}\leq 1,\text{ where }b=\frac{2\sqrt{q}}{q+1}.
	\end{equation}
	In particular, $\sigma_{2}(\mathcal{L})$ degenerates into the line segment $[1-b,1+b]$.}

The spherical Fourier transform $\hat{f}$ of a finitely supported radial function $f$  is defined by
\begin{equation}\label{Equation 2.6}
	\hat{f}(z)=\sum\limits_{x\in\x}f(x)\phi_z(x)\quad\text{ where }z\in\C.
\end{equation}
The symmetric properties of the spherical function implies that $\hat{f}$ is even and $\tau$-periodic on $\mathbb{C}$. For $1<p\leq 2$, we define the space  $\mathcal{S}_p(\mathfrak{X})$ which consists of all those functions $f$ defined on $\mathfrak{X}$ such that
\begin{equation}\label{Equation 2.7}
\nu_{m}(f)=\sup\limits_{x\in\mathfrak{X}}(1+|x|)^{m}q^{|x|/p}|f(x)|<\infty,\text{ for all }m\in\mathbb{N}.
\end{equation}
It is known that $\mathcal{S}_p(\mathfrak{X})$ form Fr\'{e}chet space with respect to these countable seminorms $\nu_{m}(\cdot)$ and are also known as the $p$-Schwartz spaces of rapidly decreasing functions on $\mathfrak{X}$ (see \cite{CMS}). For $1<p\leq 2$, we also define the space $\mathcal{H}(S_{p})^{\#}$ of all even, $\tau$-periodic function $g$ on $S_{p}$ which are holomorphic on $S_{p}^{\circ}$, continuous on $\partial S_{p}$ and satisfies
\begin{equation}\label{Equation 2.8}
\mu_{m}(g)=\sup\limits_{z\in S_{p}}\left|\frac{d^m}{dz}g(z)\right|<\infty,\text{ for all }m\in\mathbb{N}.
\end{equation}
It was  proved in \cite{JF} that the spherical Fourier transform is a topological isomorphism from $\mathcal{S}_2(\mathfrak{X})^{\#}$ onto $\mathcal{H}(S_{2})^{\#}$. In fact a similar result also holds when we consider $1<p<2$. The proof is given in appendix.

 For $1<p\leq 2$, a linear functional $T:\mathcal{S}_p(\mathfrak{X})\rightarrow\mathbb{C}$ is said to be a $L^{p}$-tempered distribution if $\langle T,f_{n}\rangle\rightarrow 0$ whenever $\nu_{m}(f_{n})\rightarrow 0$ for all $m\in\mathbb{N}$. The distribution $T$ is said to be radial if
$$\langle T,f\rangle=\langle T,\varepsilon f\rangle,\text{ for all }f\in\mathcal{S}_p(\mathfrak{X}).$$
In fact the radial part of an $L^{p}$-tempered distribution $T$ is again an $L^{p}$-tempered distribution defined by
$$\langle\varepsilon T,f\rangle=\langle T,\varepsilon f\rangle,\text{ for all }f\in\mathcal{S}_p(\mathfrak{X}).$$
The left translation $\tau_{x}$ of $T$ by an element $x\in G$ is defined as follow if $f\in\mathcal{S}_p(\mathfrak{X})$ then
$$\langle \tau_{x}T,f\rangle=T(\tau_{x^{-1}}f)=T\ast f^{\ast}(x^{-1})$$
where $f^{\ast}(x)=f(x^{-1})$.
Finally, the spherical Fourier transform $\hat{T}$ of a radial $L^{p}$-tempered distribution $T$ is a linear functional on $\mathcal{H}(S_{p})^{\#}$ defined by the following rule:
$$\langle \hat{T},\psi\rangle=\langle T,f\rangle,\text{ where }\psi\in\mathcal{H}(S_{p})^{\#}\text{ and }\hat{f}=\psi.$$

\section{Proof of Theorem A and theorem B}\label{Section 3}
Our approach is motivated by proof given in \cite{KRS2}, which in turn is influenced by Strichartz's approach. In both the papers, the Fourier transform of a tempered distribution played an important role. Proof of Theorem A and Theorem B is an immediate consequence of following three key results, namely, Lemma A, Lemma B1, and Lemma B2.
\begin{lemmaA*}\label{Theorem10}
Let $\{T_{k}\}_{k\in\mathbb{Z}}$ be a doubly infinite sequence of $L^2$-tempered distributions on $\mathfrak{X}$ satisfying,
\begin{enumerate}
\item $\mathcal{L}T_{k}=z_{0}T_{k+1}$ for some non-zero $z_{0}\in\mathbb{C}$ and
\item $|\langle T_k,\phi\rangle|\leq M\nu(\phi)$ for all $\phi\in\mathcal{S}_2(\mathfrak{X})$, where $\nu$ is some fixed semi-norm of $\mathcal{S}_2(\mathfrak{X})$ and $M>0$.
\end{enumerate}
Then we have the following results.
\begin{enumerate}
\item[(a)] If $|z_{0}|\in [1-b,1+b]$, then $\mathcal{L}T_{0}=|z_{0}|T_{0}$ and
\item[(b)] If $|z_{0}|\notin [1-b,1+b]$, then $T_k=0$ for all $k\in\mathbb{Z}$,
\end{enumerate}
where $b=\frac{2\sqrt{q}}{q+1}$.
\end{lemmaA*}

\begin{proof} We first prove part (a) of the theorem with an additional assumption that the distributions $T_{k}$ are radial.
Fix $z_{0}\in\mathbb{C}$ such that $|z_{0}|\in[1-b,1+b]$. Then $z_{0}=\gamma(\alpha)e^{i\theta}$ for a unique $\alpha\in[0,\tau/2]$ where $\theta=\arg z_{0}$.
It follows from hypothesis (1) of the theorem that $\mathcal{L}^{k}T_{0}=e^{ik\theta}\gamma(\alpha)^{k}T_{k}$ for every $k\in\mathbb{Z}$. This implies
\begin{equation}\label{eq23}
	\widehat{T_{0}}=e^{ik\theta}\left(\frac{\gamma(\alpha)}{\gamma(\cdot)}\right)^{k}\widehat{T_{k}}.
\end{equation}
 Let $\phi\in\mathcal{H}(S_{2})^{\#}$ be such that $\supp(\phi)\subseteq[-\tau/2,-\alpha-r]\cup[\alpha+r,\tau/2]$ where $r>0$. Observing the fact that $\gamma(\alpha)^{k}/\gamma(\cdot)^{k}\phi\in\mathcal{H}(S_{2})^{\#}$ and using hypothesis (2) of the theorem, we have

$$|\langle\widehat{T_{0}},\phi\rangle|=\left|\left\langle\widehat{T_{k}},e^{ik\theta}\left(\frac{\gamma(\alpha)}{\gamma(\cdot)}\right)^{k}\phi\right\rangle\right|=\left|\left\langle T_{k},\left(\left(\frac{\gamma(\alpha)}{\gamma(\cdot)}\right)^{k}\phi\right)^{\vee}\right\rangle\right|
	\leq M\nu\left[\left(\left(\frac{\gamma(\alpha)}{\gamma(\cdot)}\right)^{k}\phi\right)^{\vee}\right].$$

By the isomorphism theorem \ref{lp isom}, there exists  a fixed seminorm $\mu$ on $\mathcal{H}(S_{2})^{\#}$ such that
$$\nu\left[\left(\left(\frac{\gamma(\alpha)}{\gamma(\cdot)}\right)^{k}\phi\right)^{\vee}\right]
	\leq C\mu\left[\left(\frac{\gamma(\alpha)}{\gamma(\cdot)}\right)^{k}\phi\right]=\sup\limits_{\alpha+r\leq|s|\leq\tau/2}\left|\frac{d^{m}}{ds^{m}}\left(\left(\frac{\gamma(\alpha)}{\gamma(s)}\right)^{k}\phi\right)\right|\rightarrow 0\text{ as }k\rightarrow\infty.$$
 By similar argument as above and letting $k\rightarrow -\infty$, we can show that $\langle\widehat{T_{0}},\phi\rangle=0$ for every $\phi\in\mathcal{H}(S_{2})^{\#}$ with $\supp(\phi)\subseteq[-\alpha+r,\alpha-r]$. We proved that for any $r>0$ and for every $\phi\in\mathcal{H}(S_{2})^{\#}$ such that $\supp(\phi)\subseteq[-\tau/2,-\alpha-r]\cup[-\alpha+r,\alpha-r]\cup[\alpha+r,\tau/2]$, $\langle\widehat{T_{0}},\phi\rangle=0$.

We now show that
\begin{equation}\label{eq24}
	(\mathcal{L}-\gamma(\alpha))^{N+1}T_{0}=0\quad\text{for some }N\in\mathbb{Z}_{+}
\end{equation}
In view of the fact that the spherical transform is an isomorphism from $\mathcal{S}_2(\mathfrak{X})^{\#}$ onto $\mathcal{H}(S_{2})^{\#}$, it is enough to prove that
\begin{equation}\label{eq25}
	(\gamma(\alpha)-\gamma(s))^{N+1}\widehat{T_{0}}=0\quad\text{for some }N\in\mathbb{Z}_{+}.
\end{equation}

Let $g$ be an infinitely differentiable even function on $\mathbb{R}$ such that $g\equiv 1$ on $[-1/2,1/2]$ and $\supp(g)\subseteq(-1,1)$. Define
$$\psi_{\epsilon}(s)=\begin{cases}
g((s-\alpha)/\epsilon) & s\in[0,\tau/2]\\
g((-s-\alpha)/\epsilon) & s\in[-\tau/2,0].
\end{cases}
$$
Here $\epsilon$ is suitably chosen positive number such that $\psi_{\epsilon}\in\mathcal{H}(S_{2})^{\#}$ with $\supp(\psi_{\epsilon})\subseteq(-\alpha-\epsilon,-\alpha+\epsilon)\cup(\alpha-\epsilon,\alpha+\epsilon)$.

Note that if $\phi\in\mathcal{H}(S_{2})^{\#}$, then $(\gamma(\alpha)-\gamma(\cdot))^{N+1}\phi(1-\psi_{\epsilon})\in\mathcal{H}(S_{2})^{\#}$ with its support inside $[-\tau/2,-\alpha-\epsilon/2]\cup[-\alpha+\epsilon/2,\alpha-\epsilon/2]\cup[\alpha+\epsilon/2,\tau/2]$. and using the result proved in step 1, we have
\begin{align}
	|\langle(\gamma(\alpha)-\gamma(s))^{N+1}\widehat{T_{0}},\phi\rangle|	 &\leq|\langle\widehat{T_{0}},(\gamma(\alpha)-\gamma(s))^{N+1}\phi(1-\psi_{\epsilon})\rangle|+|\langle\widehat{T_{0}},(\gamma(\alpha)-\gamma(s))^{N+1}\phi\psi_{\epsilon}\rangle|\nonumber\\
	&=|\langle\widehat{T_{0}},(\gamma(\alpha)-\gamma(s))^{N+1}\phi\psi_{\epsilon}\rangle|\nonumber\\
	&\leq M\nu\left[\left((\gamma(\alpha)-\gamma(s))^{N+1}\phi\psi_{\epsilon}\right)^{\vee}\right]\nonumber\\
&\leq M\mu\left[(\gamma(\alpha)-\gamma(s))^{N+1}\phi\psi_{\epsilon}\right]\nonumber\\
&=M\sup\limits_{s\in[-\tau/2,\tau/2]}\left|\frac{d^{m}}{ds^{m}}\left((\gamma(\alpha)-\gamma(s))^{N+1}\phi(s)\psi_{\epsilon}(s)\right)\right|\nonumber\\
	&=M\sup\limits_{\alpha-\epsilon\leq |s|\leq\alpha+\epsilon}\left|\frac{d^{m}}{ds^{m}}\left((\gamma(\alpha)-\gamma(s))^{N+1}\phi(s)\psi_{\epsilon}(s)\right)\right|\nonumber\\
	&\leq M\sum\limits_{i=0}^{m}\binom{m}{i}\sup\limits_{\alpha-\epsilon\leq |s|\leq\alpha+\epsilon}\left|\frac{d^{i}}{ds^{i}}\left((\gamma(\alpha)-\gamma(s))^{N+1}\right)\right|\nonumber\\
	&\hspace*{5cm}\times\sup\limits_{\alpha-\epsilon\leq |s|\leq\alpha+\epsilon}\left|\frac{d^{m-i}}{ds^{m-i}}\left(\phi(s)\psi_{\epsilon}(s)\right)\right|.\label{eq26}
\end{align}
Choose $N$ large enough e.g. $N=10m+1$. Then for every $s\in(\alpha-\epsilon,\alpha+\epsilon)$ we have the following estimates:
\begin{enumerate}
	\item[(i)] $\left|\frac{d^{i}}{ds^{i}}\left((\gamma(\alpha)-\gamma(s))^{10m+2}\right)\right|\leq B_{m}|\gamma(\alpha)-\gamma(s)|^{10m+2-i}$ where $0\leq i\leq m$ and
	\item [(ii)] $\left|\frac{d^{m-i}}{ds^{m-i}}\left(\phi(s)\psi_{\epsilon}(s)\right)\right|\leq C_{m,\phi}/\epsilon^{m-i}.$
\end{enumerate}
The above estimates together with (\ref{eq26}) implies that
\begin{align*}
	\mu\left[(\gamma(\alpha)-\gamma(s))^{N+1}\phi\psi_{\epsilon}\right]&\leq M_{m,\phi}\sum\limits_{i=0}^{m}\sup\limits_{\alpha-\epsilon\leq |s|\leq\alpha+\epsilon}|\gamma(\alpha)-\gamma(s)|^{10m+2-i}\frac{1}{\epsilon^{m-i}}\\
	&\leq D\epsilon^{9m+2}\rightarrow 0\text{ as }\epsilon\rightarrow 0.
\end{align*}
This proves (\ref{eq25}).
Using the same argument given in \cite{KRS2} one can easily prove that  $$(\mathcal{L}-\gamma(\alpha))T_{0}=0.$$
This prove part (a) for radial distribution.
Now we shall prove the result for general case. To avoid triviality, we further assume that $T_{k}$ is nonzero for some (and hence for all) $k\in\mathbb{Z}$.

Observe that for any $L^{2}$-tempered distribution $T$, if $\varepsilon(\tau_{x}T)=0$ for every $x\in\mathfrak{X}$ then $T=0$. Indeed the above assumption on $T$ implies that $$\langle\tau_{x}T,\delta_{0}\rangle=T\ast\delta_{0}(x^{-1})=0$$ for all $x\in\mathfrak{X}$, where $\delta_{0}$ denotes the Dirac-Delta function at $o$. Since $T\ast\delta_{0}=T$ in the sense of distribution thus $T=0$.
This shows that for every $k\in\mathbb{Z}$, there exists an $x\in\mathfrak{X}$ such that the distribution $\tau_{x}T_{k}$ has a nonzero radial part.

Now we claim that if $\varepsilon(\tau_{x}T_{0})\neq 0$ for some $x\in\mathfrak{X}$, then $\varepsilon(\tau_{x}T_{k})\neq 0$ for every $k\in\mathbb{Z}$. To prove this it is enough to show that if $\varepsilon(\tau_{x}T_{0})\neq 0$ for some $x\in\mathfrak{X}$, then $\varepsilon(\tau_{x}T_{-1})\neq 0$ and $\varepsilon(\tau_{x}T_{1})\neq 0$.
If $\varepsilon(\tau_{x}T_{-1})=0$ then $\mathcal{L}\varepsilon(\tau_{x}T_{-1})=0$. Since $\mathcal L$ commutes with translation and radialization, and  $\mathcal{L}T_{-1}=z_0T_{0}$ for $z_0\neq 0$ thus $\varepsilon(\tau_{x}T_{0})=0$.
On the other hand if  $\varepsilon(\tau_{x}T_{1})=0$ then $$\langle\tau_{x}T_{1},\phi\rangle=\langle\tau_{x}\mathcal{L}T_{0},\phi\rangle=\langle\tau_{x}T_{0},\mathcal{L}\phi\rangle=0$$ for every $\phi\in\mathcal{S}_2(\mathfrak{X})^{\#}$. Since $\gamma(s)^{-1}\hat{\psi}(s)\in\mathcal{H}(S_{2})^{\#}$ for every $\psi\in\mathcal{S}_2(\mathfrak{X})^{\#}$.
Thus $\psi$ can be written as $\psi=\mathcal{L}\phi$ for some $\phi\in\mathcal{S}_2(\mathfrak{X})^{\#}$. Hence $\tau_{x}T_{0}=0$ for all $x\in\x$. This proves our claim.

It is easy to show that for every $x\in\mathfrak{X}$, the sequence $\{\varepsilon(\tau_{x}T_{k})\}$ of radial distributions satisfies the hypothesis of this theorem.
 Since the result is already proved for radial $L^{2}$-tempered distributions, we have
$$\mathcal{L}\varepsilon(\tau_{x}T_{0})=|z|\varepsilon(\tau_{x}T_{0})\quad\text{for all }x\in\mathfrak{X}.$$
Therefore $\varepsilon(\tau_{x}(\mathcal{L}T_{0}-|z|T_{0}))=0$ for all $x\in\mathfrak{X}$. From above observation we have $\mathcal{L}T_{0}=|z|T_{0}$. This complete the proof of part (a).

We shall prove part (b) of the theorem only for radial case. The proof for the general case follows in a similar way as in part (a). Now assuming that $T_{k}$ are radial, we have for any $\phi\in\mathcal{S}_2(\mathfrak{X})^{\#}$,

$$	|\langle\widehat{T_{0}},\phi\rangle|=\left|\left\langle\widehat{T_{k}},\left(\frac{z_{0}}{\gamma(s)}\right)^{k}\phi\right\rangle\right|
	\leq M\nu\left[\left(\left(\frac{z_{0}}{\gamma(s)}\right)^{k}\phi\right)^{\vee}\right]
	\leq M\mu\left[\left(\frac{z_{0}}{\gamma(s)}\right)^{k}\phi\right].$$
If $|z_{0}|<\gamma(s)$ (resp. $|z_{0}|>\gamma(s)$) for $s\in[-\tau/2,\tau/2]$, then letting $k\rightarrow\infty$ (resp. $k\rightarrow -\infty$) in the above equation we conclude that $\langle T_{0},\phi\rangle=0$ for all $\phi\in\mathcal{S}_2(\mathfrak{X})^{\#}$. This completes the proof.
\end{proof}
Now we consider the case $1<p<2$. The main difference from the above lemma and the classical Euclidean case is that the $L^p$-tempered distribution acts on holomorphic functions. Therefore the main technique of the previous lemma, namely, the use of function whose Fourier transform are supported outside of an interval will not work.

\begin{lemmaB*}\label{Theorem11}
For $1<p<2$, let $\{T_{k}\}_{k\in\mathbb{Z}^+}$ be an infinite sequence of $L^p$-tempered distributions on $\mathfrak{X}$ satisfying,
\begin{enumerate}
\item $\mathcal{L}T_{k}=\lambda T_{k+1}$ for some non-zero $\lambda\in\mathbb{C}$ and
\item $|\langle T_k,\phi\rangle|\leq M\nu(\phi)$ for all $\phi\in\mathcal{S}_p(\mathfrak{X})$, where $\nu$ is some fixed semi-norm of $\mathcal{S}_p(\mathfrak{X})$ and $M>0$.
\end{enumerate}
Then we have the following results.
\begin{enumerate}
\item[(a)] If $|\lambda|=\gamma(i\delta_{p^{\prime}})$, then $\mathcal{L}T_{0}=|\lambda|T_{0}$ and
\item[(b)] If $|\lambda|<\gamma(i\delta_{p^{\prime}})$, then $T_k=0$ for all $k\in\mathbb{Z}^+$.
\item[(c)] There are solutions which are not eigen-distributions whenever $\gamma(\tau/2+i\delta_{p^{\prime}})>|\lambda|>\gamma(i\delta_{p^{\prime}})$
\end{enumerate}
\end{lemmaB*}
\begin{proof}
We prove the this result for radial distributions, while the general case follows in a similar way as in Theorem \ref{Theorem10}.
For $p\in(1,2)$ let $z_{0}=i\delta_{p^{\prime}}$.  For a fixed $N\in\mathbb{Z}_{+}$ we claim that $$(\gamma(z_{0})-\gamma(z))^{N+1}\widehat{T_{0}}=0.$$
As observed earlier, for any $\phi\in\mathcal{H}(S_{p})^{\#}$  we have,
$$ |\langle(\gamma(z_{0})-\gamma(z))^{N+1}\widehat{T_{0}},\phi\rangle|\leq
M\mu\left[\left(\left(\frac{\gamma(z_{0})}{\gamma(z)}\right)^{k}(\gamma(z_{0})-\gamma(z))^{N+1}\phi\right)\right].$$
Since $\gamma(z)$ and $\phi(z)$ are $\tau-$periodic, even functions on $S_{p}$, so the seminorm $\mu$ on $\mathcal{H}(S_{p})^{\#}$ is given by
\begin{align*}
\mu\left[\left(\left(\frac{\gamma(z_{0})}{\gamma(z)}\right)^{k}(\gamma(z_{0})-\gamma(z))^{N+1}\phi\right)\right]&=\sup\limits_{z\in S_{p}}\left|\frac{d^{m}}{dz^{m}}\left(\left(\frac{\gamma(z_{0})}{\gamma(z)}\right)^{k}(\gamma(z_{0})-\gamma(z))^{N+1}\phi(z)\right)\right|\\
&=\sup\limits_{z\in S_{p}^{+}}\left|\frac{d^{m}}{dz^{m}}\left(\left(\frac{\gamma(z_{0})}{\gamma(z)}\right)^{k}(\gamma(z_{0})-\gamma(z))^{N+1}\phi(z)\right)\right|\\
&=\sup\limits_{z\in S_{p}^{+}} F_{k}(z)\text{ (say),}
\end{align*}
where $S_{p}^{+}=\{z\in S_{p}:~\; |\Re z|\leq \tau/2\;\; \text{and}\;\;  \Im z\geq 0\}$.

To prove our claim it is enough to show that $\sup\limits_{z\in S_{p}^{+}} F_{k}\rightarrow 0$ as $k\rightarrow\infty$. Now for all $z\in S_{p}^{+}$,
\begin{align}
F_{k}(z)&\leq\sum\limits_{i=0}^{m}\binom{m}{i}\left|\frac{d^{i}}{dz^{i}}\left(\left(\frac{\gamma(z_{0})}{\gamma(z)}\right)^{k}\right)\right|
\left|\frac{d^{m-i}}{dz^{m-i}}\left((\gamma(z_{0})-\gamma(z))^{N+1}\phi(z)\right)\right|\nonumber\\
&\leq\sum\limits_{i=0}^{m}\binom{m}{i}B_{m}\left|\frac{\gamma(z_{0})}{\gamma(z)}\right|^{k}k(k+1)(k+2)\ldots(k+i-1)\label{eq30}\\
&\hspace*{4.8cm}\times\left|\frac{d^{m-i}}{dz^{m-i}}\left((\gamma(z_{0})-\gamma(z))^{N+1}\phi(z)\right)\right|.\nonumber
\end{align}
From above inequality we also have
\begin{equation}F_{k}(z)\leq A_{m}k^{m}\left|\frac{\gamma(z_{0})}{\gamma(z)}\right|^{k}\label{eq31}
\end{equation}

where
$A_{m}=\max\limits_{0\leq i\leq m}\left[\sup\limits_{z\in S_{p}^{+}}\left|\frac{d^{i}}{dz^{i}}\left((\gamma(z_{0})-\gamma(z))^{N+1}\phi(z)\right)\right|\right].$

 In the above calculation we have also used the following estimates. For every $z\in S_{p}^{+}$,
\begin{enumerate}
	\item[(i)] $\left|\frac{\gamma(z_{0})}{\gamma(z)}\right|\leq 1$,
	\item[(ii)] $\left|\frac{d^{i}}{dz^{i}}\left(\left(\frac{\gamma(z_{0})}{\gamma(z)}\right)^{k}\right)\right|\leq B_{m}\left|\frac{\gamma(z_{0})}{\gamma(z)}\right|^{k}k(k+1)(k+2)\ldots(k+i-1)$, where $0\leq i\leq m$.
\end{enumerate}

In order to prove that $\sup\limits_{z\in S_{p}^{+}} F_{k}\rightarrow 0$ as $k\rightarrow\infty$, it is enough to show that
\begin{enumerate}
	\item[(a)]  $\sup\limits_{z\in V_{k}^{c}} F_{k}\rightarrow 0$,
	\item[(b)] $\sup\limits_{z\in V_{k}} F_{k}\rightarrow 0$ as $k\rightarrow\infty$.
\end{enumerate}
 where for $k$ (large enough)
$$V_{k}=\{z\in S_{p}^{+}: |\Re z|<(k^{1/4}\log q)^{-1}\text{ and } \delta_{p}-\log(1+1/k^{1/6})(\log q)^{-1}<\Im z\leq\delta_{p}\}.$$
First we deal with (a). In view of equation (\ref{eq31}), it is sufficient to show that for every $z\in V_{k}^{c}$ there exists a constant $c>0$ such that
\begin{equation}\label{eq32}
\left|\frac{\gamma(z_{0})}{\gamma(z)}\right|^{k}\leq\left(1+\frac{c}{\sqrt{k}}\right)^{-k}.
\end{equation}

First case, if $0\leq \Im z\leq \delta_{p}-\frac{\log(1+1/k^{1/6})}{\log q}$ then $|\gamma(z)|\geq |\gamma(i(\delta_{p}-\frac{\log(1+1/k^{1/6})}{\log q}))|$.
Hence
\begin{align}
	|\gamma(z)|-|\gamma(z_{0})|&\geq \left(\frac{q^{1/p}+q^{1/p^{\prime}}}{q+1}\right)-\left(\frac{q^{1/p}(1+k^{-1/6})^{-1}+q^{1/p^{\prime}}(1+k^{-1/6})}{q+1}\right)\nonumber\\
	&=\frac{1}{q+1}\frac{q^{1/p^{\prime}}}{k^{1/6}(1+k^{1/6})}\left[k^{1/6}(q^{2/p-1}-1)-1\right].\label{eq33}
\end{align}
Since $q^{2/p-1}-1>0$, there exists $k_{0}\in\mathbb{N}$ such that $k^{1/6}(q^{2/p-1}-1)\geq 2$ for every $k\geq k_{0}$. This together with (\ref{eq33}) gives the desired inequality (\ref{eq32}).

Other case if $z\in V_{k}^{c}$ is such that $\tau/2\geq|\Re z|\geq (k^{1/4}\log q)^{-1}$, then
$$|\gamma(z)|\geq \left(1-\frac{q^{1/p^{\prime}}+q^{1/p}}{q+1}cos(k^{-1/4})\right)$$
and
$$|\gamma(z)|-|\gamma(z_{0})|\geq\left(\frac{q^{1/p^{\prime}}+q^{1/p}}{q+1}(1-\cos(k^{-1/4}))\right)\geq \frac{c}{k^{1/2}}.$$
Thus for every $z\in V_{k}^{c}$, inequality (\ref{eq32}) holds. Eventually,
$\sup\limits_{z\in V_{k}^{c}} F_{k}\rightarrow 0$ as $k\rightarrow\infty$.

 Now let us assume that $z\in V_{k}$. Then $z=a+i\delta_{r}$, where $|a|<(k^{1/4}\log q)^{-1}$ and $\delta_{p}-\log(1+1/k^{1/6})(\log q)^{-1}<\delta_{r}\leq\delta_{p}$ and we have
\begin{align*}
	|\gamma(z)-\gamma(z_{0})|^{2}&=\left(\frac{q^{1/p^{\prime}}+q^{1/p}}{q+1}-\frac{q^{1/r^{\prime}}+q^{1/r}}{q+1}\cos(a\log q)\right)^{2}
	+\left(\frac{q^{1/r}-q^{1/r^{\prime}}}{q+1}\right)^{2}\sin^{2}(a\log q)\\
	&\leq \left(\frac{q^{1/p}+q^{1/p^{\prime}}}{q+1}-\frac{q^{1/p}(1+k^{-1/6})^{-1}+q^{1/p^{\prime}}(1+k^{-1/6})}{q+1}\cos(k^{-1/4})\right)^{2}\\
	&\hspace*{8cm}+\left(\frac{q^{1/p}-q^{1/p^{\prime}}}{q+1}\right)^{2}\sin^{2}(k^{-1/4})
\end{align*}
It follows from the inequality $\sqrt{|x|^2+|y|^2}\leq |x|+|y|$ that
\begin{align*}
	|\gamma(z)-\gamma(z_{0})|&\leq \left(\frac{q^{1/p}+q^{1/p^{\prime}}}{q+1}-\frac{q^{1/p}(1+k^{-1/6})^{-1}+q^{1/p^{\prime}}(1+k^{-1/6})}{q+1}\cos(k^{-1/4})\right)+\frac{c_{1}}{k^{1/4}}\\
	&= \left(\frac{q^{1/p}+q^{1/p^{\prime}}}{q+1}-\frac{q^{1/p}(1+k^{-1/6})^{-1}+q^{1/p^{\prime}}(1+k^{-1/6})}{q+1}\right)\\
	&\hspace*{3cm}+\frac{q^{1/p}(1+k^{-1/6})^{-1}+q^{1/p^{\prime}}(1+k^{-1/6})}{q+1}(1-\cos(k^{-1/4}))+\frac{c_{1}}{k^{1/4}}\\
	&\leq \frac{c_{3}}{k^{1/6}}+\frac{c_{2}}{k^{1/2}}+\frac{c_{1}}{k^{1/4}}\leq\frac{c}{k^{1/6}},
\end{align*}
where the constants $c_{1},c_{2},c_{3}$ (and hence $c$) are independent of $k$. If we take $N=7m+1$ then each term in equation (\ref{eq30}) atleast contains the factor $(\gamma(z_{0})-\gamma(z))^{N+1-m}=(\gamma(z_{0})-\gamma(z))^{6m+2}$.
Thus from above estimates and equation (\ref{eq30}), we finally have $\sup\limits_{z\in V_{k}} F_{k}\leq \frac{C}{k^{1/3}},$ where $C$ is independent of $k$ and $\sup\limits_{z\in S_{p}} F_{k}\rightarrow 0$ as $k\rightarrow\infty$.  Using the same argument given in \cite{KRS2} one can easily prove that $N=0$. This completes the proof of part (a) for radial and eventually for general distributions.

The proof of part (b) is similar to that of part (b) Lemma A. To prove part (c), assume that $\gamma(\tau/2+i\delta_{p^{\prime}})>|\lambda|>\gamma(i\delta_{p^{\prime}})$. Then $\gamma(S_{p}^{\circ})$ intersects $\{w\in\C:|w|=|\lambda|\}$ at infinitely many points. Let $p<q<r<2$ be such that $\gamma(\alpha+i\delta_{q^{\prime}})e^{-i\theta_1}=\gamma(\beta+i\delta_{r^{\prime}})e^{-i\theta_2}=\lambda$ for some $\theta_1,\theta_2\in(0,2\pi)$. If we define $T_{k}=e^{ik\theta_1}\phi_{\alpha+i\delta_{q^{\prime}}}+e^{ik\theta_2}\phi_{\beta+i\delta_{r^{\prime}}}$ where $k\in\mathbb Z_{+}$, then $T_k$ satisfies all the hypothesis of Lemma B1 but $T_0$ fails to be an eigen-distribution of $\mathcal L$.
\end{proof}
Now we state the another important lemma whose proof is just a repetation of the arguments of the Lemma B1.

\begin{lemmaB1*}\label{Theorem12}
For $1<p<2$, let $\{T_{-k}\}_{k\in\mathbb{Z}^+}$ be an infinite sequence of $L^p$-tempered distributions on $\mathfrak{X}$ satisfying,
\begin{enumerate}
\item $\mathcal{L}T_{-k}=\lambda T_{-k+1}$ for some non-zero $z\in\mathbb{C}$ and
\item $|\langle T_{-k},\phi\rangle|\leq M\nu(\phi)$ for all $\phi\in\mathcal{S}_p(\mathfrak{X})$, where $\nu$ is some fixed semi-norm of $\mathcal{S}_p(\mathfrak{X})$ and $M>0$.
\end{enumerate}
Then we have the following results.
\begin{enumerate}
\item[(a)] If $|\lambda|=\gamma(\tau/2+i\delta_{p^{\prime}})$, then $\mathcal{L}T_{0}=|z|T_{0}$ and
\item[(b)] If $|\lambda|>\gamma(\tau/2+i\delta_{p^{\prime}})$, then $T_{-k}=0$ for all $k\in\mathbb{Z}^+$.
\item[(c)] There are solutions which are not eigen-distributions whenever $\gamma(\tau/2+i\delta_{p^{\prime}})> |\lambda|>\gamma(i\delta_{p^{\prime}})$
\end{enumerate}
\end{lemmaB1*}

\noindent {\it Proof of Theorem A and Theorem B:} Now we will conclude the rest of the proof of the our main theorems. Let $z_0=\gamma(z)$ for $z\in\R\setminus\tau/2\Z$. From the give assumption in Theorem A if we define $T_k=\gamma(z)^{-k}\mathcal L^k f$ then $T_k$ is a $L^2$-tempered distribution and the proof of the Theorem A is a consequence of the Lemma A.
Similarly if we assume that $T_{k}=\gamma(i\delta_{p^{\prime}})^{-k}\mathcal{L}^{k}f$ where $k\in\mathbb{Z}_{+}$. According to the hypothesis of Theorem B, $\|T_{k}\|_{p^\prime,\infty}\leq M$	 for all $k\in\mathbb{Z}_{+}$. It is easy to show that each $T_{k}$ is an $L^p$-tempered distribution which satisfies all the hypothesis of Lemma B1. Hence $\mathcal{L}f=\gamma(i\delta_{p^{\prime}})f$. This completes the proof of the 1st part of Theorem B.
The remaining part of Theorem B will follow from  the Lemma B2.
\section{Sharpness of the main results and the case $q=1$}
\begin{enumerate}
\item Observe that if $\mathcal L u=\gamma(z)u$ for $z\in(\tau/2)\mathbb{Z}$ then $u\notin L^{2,\infty}(\mathfrak{X})$. Assume $u(x_o)\neq 0$ for some $x_o\in \mathfrak X$.
If $u\in L^{2,\infty}(\mathfrak{X})$ then  $f(x)=\int_Ku(x_okx)dk=u(x_o)\phi_z(x)$ also belong to $L^{2,\infty}(\mathfrak{X})$. In view of Lemma \ref{phiz} this not true as
	$\phi_{z}$ does not belong to $L^{2,\infty}(\mathfrak{X})$ whenever $z\in(\tau/2)\mathbb{Z}$. This observation shows that Theorem A is no longer valid for any  $z\in(\tau/2)\mathbb{Z}$. However if we replace the $L^{2,\infty}$ estimate by $\|\phi_{0}^{-1}\mathcal{L}^{k}f\|_{L^{\infty}(\mathfrak{X})}\leq M|\gamma(z)|^{k}$ for all $k\in\mathbb{Z}$, then Theorem A holds true.
	\item It follows from equation (\ref{eqsf}) and Lemma \ref{phiz} that in Theorem A (resp. in Theorem B), the $L^{2,\infty}$ (resp. the $L^{p^{\prime},\infty}$) estimate cannot be replaced by $L^{2,r}$, $r<\infty$ (resp. $L^{q^{\prime},r}$ with $q>p$ or $L^{p^{\prime},r}$ with $r<\infty$). The proof is similar as above.
	\item  Unlike Theorem B, it is necessary to consider all integral powers of $\mathcal{L}$ in Theorem A. Otherwise for $z\in\R\setminus(\tau/2)\Z$, if we choose $s_{1},s_{2}\in\R\setminus(\tau/2)\Z$ such that $\gamma(s_{i})\leq\gamma(z)$ for each $i=1,2$ and define $f=\phi_{s_{1}}+\phi_{s_{2}}$ then $f$ satisfies all the hypothesis of Theorem A but $\mathcal{L}f\neq\gamma(z)f$.
	\item The conclusions of Theorem B, Part $1$ (resp. Part $2$) does not hold for $z=\alpha\pm i\delta_{p^{\prime}}$ where $\alpha\in\mathbb{R}\setminus (n\tau)\mathbb{Z}$ (resp. $\alpha\in\mathbb{R}\setminus ((2n+1)\tau/2)\mathbb{Z}$). The counterexamples can be constructed in a similar way as we did in Lemma B1, part (c).
    \item In Theorem B (resp. in Theorem A), if we consider $z=\alpha\pm i\delta_{p^{\prime}}$ where $\alpha\in\mathbb{R}$ (resp. $z\in\R$) and substitute the $L^{p^{\prime},\infty}$ norm (resp. $L^{2,\infty}$) by the $L^{q^{\prime},r}$ where $1\leq q<p<2$ (resp. $L^{p^{\prime},r}$ where $1<p<2$, $1\leq r\leq\infty$ or $L^{\infty}$), then there are functions which satisfies the hypothesis but are not eigenfunctions of $\mathcal L$.
\end{enumerate}
\subsection{Results on $\mathbb{Z}$}\label{Section 5}
Having dealt with all viable generalizations of Roe's result on homogeneous trees of degree $q+1$ where $q\geq 2$, it becomes quite plausible to consider this problem on a homogeneous tree of degree 2 (i.e., when $q=1$), which may be identified to $\mathbb{Z}$. This identification distinguishes it from others in terms of its geometric and analytic properties. However, the most intriguing difference in the context of this article lies in the spectrum of their respective Laplace operators. Hence we deal with this case separately. The Laplace operator on $\mathbb{Z}$ is defined as
$$\mathcal L_{\mathbb{Z}} f(m)=f(m)-\frac{f(m-1)+f(m+1)}{2}\quad\text{for all }m\in\mathbb{Z}.$$
Unlike the spectrum of $\mathcal{L}$ which is a $p$-depend on the elliptic region, the spectrum of $\Delta_{\mathbb{Z}}$ is always a line segment. Having said that, it must also be noted that the term $-2\Delta_{\mathbb{Z}}$ is in some way, a `discrete' representation of the operator $-d^{2}/dx^{2}$ whose spectrum is also a line segment in $\mathbb{R}$. These observations allowed us to use the Fourier transform techniques and motivated us to emphasize an analogy with Roe's result on $\mathbb{R}$.
The Fourier transform $\widehat{f}$ of a finitely supported function $f$ defined on $\mathbb{Z}$, is a function on $\mathbb{T}$ given by
$$\widehat{f}(s)=\sum\limits_{m\in\mathbb{Z}}f(m)e^{ims}\quad\text{ where }f(m)=\int\limits_{-\pi}^{\pi}\widehat{f}(s)e^{ims}ds$$
represents it's inverse Fourier transform. Let us define the Schwartz space $\mathcal{S}(\mathbb{Z})=\{f:\mathbb{Z}\rightarrow\mathbb{C}:\lambda_{n}(f)<\infty\text{ for all }n\in\mathbb{N}\}$, equipped with the countable family of semi-norms $\lambda_{n}(f)=\sup(1+|m|)^{n}|f(m)|$. It is easy to see that the map $f\rightarrow\widehat{f}$ is a topological isomorphism from $\mathcal{S}(\mathbb{Z})$ onto $\mathcal{C}^{\infty}(\mathbb T)$ where $\mathcal{C}^{\infty}(\mathbb T)=\{g:\mathbb{R}\rightarrow\mathbb{C}:g\text{ is infintely differentiable on }\mathbb{R},~g(x+\pi)=g(x)\text{ and }\mu_{l}(g)<\infty\text{ for all }x\in\mathbb{R},l\in\mathbb{N}\text{ respectively}\}$ where $\mu_{l}(g)=\sup|g^{(l)}(s)|$ defines a countable family of semi-norms. Analogous to the Euclidean case, a distribution $T$ is a continuous linear functional on $\mathcal{S}(\mathbb{Z})$ whose Fourier transform is defined as
$$\langle \widehat{T},\phi\rangle=\langle T,\left(\phi^{\vee}\right)^{\#}\rangle,\text{ where }\widehat{\phi^{\vee}}=\phi\text{ and }\left(\phi^{\vee}\right)^{\#}(m)=\phi^{\vee}(-m).$$
 Now we state the Theorem  on $\mathbb{Z}$, which  can be proved by the similar argument developed by Roe in \cite{R}.
\begin{Theorem}\label{Theorem4.1}
	Let $\{f_{k}\}_{k\in\mathbb{Z}}$ be a doubly infinte sequence of functions on $\mathbb{Z}$ which satisfies
$\mathcal{L}_{\mathbb{Z}} f_{k}=(1-\cos\alpha)f_{k+1}\quad\text{ for }\alpha\neq 0$
	and there exists constants $M_{k}\geq 0$, $\beta\in(0,1]$ and a non-negative integer $n$ such that
	$|f_{k}(m)|\leq M_{k}(1+|m|)^{n+\beta}\text{ for all }k,m\in\mathbb{Z}.$	
If $\liminf\limits_{k\rightarrow \infty}\frac{M_{k}}{(1+\epsilon)^{k}}=0$  and $\liminf\limits_{k\rightarrow \infty}\frac{M_{-k}}{(1+\epsilon)^{k}}=0$  for all $ \epsilon>0$ 	then $f_{0}(m)=p(m)e^{im\alpha}+q(m)e^{-im\alpha}$ where $p,q$ are polynomials of degree atmost $n$.
\end{Theorem}
\noindent Remark:
	We can also extend the above result to $\mathbb{Z}^{n}$ (as done by Strichartz \cite{S}) where the Laplacian is defined as
	$$\mathcal{L}_{\mathbb{Z}^{n}}f(m)=f(m)-\frac{1}{2n}\sum\limits_{k:|m-k|=1}f(k).$$
But here we assume that $M_{k}$ satisfies a sublinear growth, that is $\lim\limits_{k\rightarrow \pm\infty}\frac{M_{k}}{k}=0$.


\section{Appendix}\label{Section 6}
To make our exposition self-contained, we now prove the isomorphism theorems for the spherical Fourier transform defined on the space $\mathcal{S}_p(\mathfrak{X})^{\#}$ where $1<p\leq 2$. However, the $\mathcal{S}_2(\mathfrak{X})^{\#}$ isomorphism theorem is already proved in \cite{JF} (see Theorem 3.3).
Recalling expression (\ref{Equation 2.6}), the spherical Fourier transform of a function $f\in\mathcal{D}(\mathfrak{X})^{\#}$ can also be written as
\begin{equation}\label{eq15}
\hat{f}(z)=\sum\limits_{n\in\mathbb{Z}}\mathcal{A}f(n)q^{inz},
\end{equation}
where $\mathcal{A}f$ denotes the Abel transformation of $f$. Cowling et al. proved (see Theorem 2.5, \cite{CMS}) that $f\rightarrow\mathcal{A}f$ is a topological isomorphism from $\mathcal{S}_p(\mathfrak{X})^{\#}$ onto $q^{-\delta_{p}|\cdot|}S_{ev}(\mathbb{Z})$, for every $p\in(1,2]$ where $S_{ev}(\mathbb{Z})$ is the space of all even functions on $\mathbb{Z}$ such that $\lambda_{m}(F)=\sup\limits_{n\in\mathbb{Z}}(1+|n|)^m|F(n)|<\infty$
for all $m\in\mathbb{Z}_{+}$ and $\lambda_{m}(\cdot)$ defines a countable family of semi-norms on $S_{ev}(\mathbb{Z})$. In fact, for any natural number $m\geq 2$, there exists a constant $C(p,m)>0$ such that for all $f\in\mathcal{S}_p(\mathfrak{X})^{\#}$,
\begin{equation}\label{eq16}
C^{-1}\lambda_{(m-2)}(q^{\delta_{p}|\cdot|}\mathcal{A}f)\leq\nu_{m}(f)\leq C\lambda_{m}(q^{\delta_{p}|\cdot|}\mathcal{A}f).
\end{equation}
We use the above result to prove the following isomorphism theorem. Proof of this theorem is influenced by the technique given in \cite{CMS} and \cite{BD}.

\begin{Theorem}\label{lp isom}
The map $f\rightarrow\hat{f}$ is a topological isomorphism from $\mathcal{S}_p(\mathfrak{X})^{\#}$ onto $\mathcal{H}(S_{p})^{\#}$, for every $p\in(1,2]$.
\end{Theorem}
\begin{proof}
Fix $p\in(1,2)$ (proof of the case $p=2$ is similar). Let $f\in\mathcal{S}_p(\mathfrak{X})^{\#}$ and $z\in S_p$. Then it is clear that the infinite series (\ref{eq15}) converges uniformly on $S_p$ and consequently $\hat{f}$ is well-defined. The analyticity of $\hat{f}$ on $S_p^{\circ}$ follows directly from the analyticity of $q^{inz}$ together with the fact that the infinite series (\ref{eq15}) converges uniformly on any compact subset of $S_p^{o}$. Infact for every $m\in\mathbb{Z}_{+}$,
$$\hat{f}^{(m)}(z)=\sum\limits_{n\in\mathbb{Z}}(in\log q)^m\mathcal{A}f(n)q^{inz},\quad\text{for all } z\in S_p^{o}.$$
The above expression together with equation (\ref{eq16}) implies that for every semi-norm $\mu_{m}$ of $\mathcal{H}(S_{p})^{\#}$, there exists a semi-norm $\nu_{(m+4)}$ of $\mathcal{S}_p(\mathfrak{X})^{\#}$ such that,
$$\mu_{m}(\hat{f})\leq C\nu_{(m+4)}(f)\quad\text{ for all }f\in\mathcal{S}_p(\mathfrak{X})^{\#}.$$
Conversely, assume $g\in\mathcal{H}(S_{p})^{\#}$. Then for all $r$ with $p<r\leq 2$, the function $g(\cdot+i\delta_{r})$ is an infinitely differentiable function of period $\tau$. Hence g has a Fourier series representation of the form $g(s)=\sum\limits_{n\in\mathbb{Z}}F(n)q^{ins}$, where
$$F(n)=\frac{1}{\tau}\int\limits_{-\tau/2}^{\tau/2}g(s)q^{-ins}ds$$
yields the $n^{\text{th}}$ Fourier co-efficient of the function $g$. Our aim is to prove that $F\in q^{-\delta_{p}|\cdot|}S_{ev}(\mathbb{Z})$. Applying the Cauchy's integral theorem to $g$, it is easy to verify that for every $r\in(p,2]$ and $n\in\mathbb{Z}$,
\begin{align}
F(n)&=\frac{1}{\tau}\int\limits_{-\tau/2}^{\tau/2}g(s+i\delta_{r})q^{-in(s+i\delta_{r})}ds=\frac{1}{\tau}\int\limits_{-\tau/2}^{\tau/2}g(s-i\delta_{r})q^{-in(s-i\delta_{r})}ds.\label{eq19}
\end{align}
Infact the first equality in (\ref{eq19}) can be proved using the closed rectangle
\begin{multline*}
\Gamma(z)=\{z\in\mathbb{C}:\Im z=0,-\tau/2\leq \Re z\leq\tau/2\}\cup\{z\in\mathbb{C}: \Re z=\tau/2,0\leq \Im z\leq\delta_{r}\}\\
\cup\{z\in\mathbb{C}:\Im z=\delta_{r},\tau/2\leq \Re z\leq-\tau/2\}\cup\{z\in\mathbb{C}:\Re z=-\tau/2,\delta_{r}\leq \Im z\leq 0\}.
\end{multline*}
Using the identities (\ref{eq19}) and noting that $g$ is even, one can easily prove that $F(-n)=F(n)$ for all $n\in\mathbb{N}$, that $F$ is even in $\mathbb{Z}$. Integrating by parts the second equation in (\ref{eq19}) (m times) and further using the Dominated convergence theorem and letting $r\rightarrow p$ we have,
\begin{equation}\label{eq21}
\lambda_{m}(q^{\delta_{p}|\cdot|}F)\leq C \mu_{m}(g)\quad\text{for every }m\in\mathbb{Z}_{+}.
\end{equation}
Hence there exists an unique $f\in\mathcal{S}_p(\mathfrak{X})^{\#}$ such that $\mathcal{A}f=F$ and $g=\hat{f}$. Further using equations (\ref{eq16}) and (\ref{eq21}) we conclude that
$$\nu_{m}(f)\leq C \mu_{m}(\hat{f})\quad\text{for every }m\in\mathbb{Z}_{+}.$$
This completes the proof.
\end{proof}

 
\end{document}